\documentclass[a4,12pt,reqno]{amsart}
\usepackage{enumerate} \usepackage{lmodern}   \usepackage{euscript} \usepackage[pdftex]{hyperref}
\newcommand{\bz}{\overline{z}} \newcommand{\bw}{\overline{w}}
\newcommand{\pz}{\partial_z} \newcommand{\pbz}{\partial_{\bz}}

\numberwithin{equation}{section}  \makeatletter\@addtoreset{equation}{section}
   \DeclareMathSymbol{\subsetneqq}{\mathbin}{AMSb}{36}
\newtheorem {theorem}{Theorem}[section]            \newtheorem {lemma}[theorem]{Lemma}
   \newtheorem {corollary}[theorem]{Corollary}     \newtheorem {remark}[theorem]{Remark}
\newtheorem {proposition}[theorem]{Proposition}       
         
\newcommand{\C}{\mathbb C}      \newcommand{\R}{\mathbb R}

\begin{document}
\title[]
{Operational formulae for the complex Hermite polynomials $H_{p,q}(z,\bz )$}
\author[A. Ghanmi]{{\bf  Allal Ghanmi}}
\address{
        {Department of Mathematics, Faculty of Sciences,   \newline
        P.O. Box 1014, Mohammed V - Agdal University, 10000 Rabat, Morocco}}
        \email{ag@fsr.ac.ma}

\maketitle

\begin{abstract}
We give operational formulae of Burchnall type involving complex Hermite polynomials.
Short proofs of some known formulae are given and new results involving these polynomials,
including Nielsen's identities and Runge addition formula, are derived.
\bigskip

\end{abstract}

\section{Introduction} 

Operational formulae, involving special orthogonal polynomials, appear as useful tools to obtain some new results or to give short proofs of known formulae,
see Burchnall \cite{Burchnall41,Burchnall51}, Carlitz \cite{Carlitz60}, Al-Salam \cite{Al-Salam61}, Gould-Hopper \cite{GouldHopper62},
  Chatterjea \cite{Chatterjea63a}, 
  Singh \cite{Singh65}, 
   Karande and Thakare \cite{KarandeThakare73}, Al-Salam and Ismail \cite{Al-SalamIsmail75},
   Patil and Thakare  \cite{PatilThakare78}, Chatterjea and Srivastava \cite{ChatterjeaSrivastava93}, Purohit and  Raina \cite{PurohitRaina09}, Kaanoglu \cite{Kaanoglu12} among others.
   As basic 
   example, one cites the classical real Hermite polynomials (\cite{Hermite1864-1908,Rainville71,Szego75,Thangavelu93}):
 $$ H_{m}(x)=(-1)^{m}e^{x^2}D^m(e^{-x^2}); \qquad D:=\frac d{dx},$$
 which are extensively studied and have found wide application in various branches of mathematics, physics and technology.
 They possess a complete and rich list of remarkably interesting properties and some of them can be obtained using the Burchnall's operational formula \cite{Burchnall41}:
          \begin{align} \label{ROpfsum}
              \left(- D +2x\right)^{m}(f) = m! \sum\limits_{k=0}^{m} \frac{(-1)^{k}}{k!}
              \frac{H_{m-k}(x)}{(m-k)!} D^{k}(f).
          \end{align}
 Indeed, it can be employed to give a direct and simple proof of the Runge addition formula  \cite{Runge14,Kampe23}:
          \begin{align} \label{addition-formula}
             H_{m}(x+y)=\left(\frac 12\right)^{m/2} m! \sum\limits_{k=0}^{n}\frac{H_{k}(\sqrt{2}x)}{k!} \frac{H_{m-k}(\sqrt{2}y)}{(m-k)!} .
          \end{align}
as well as of the  quadratic recurrence formula (Nielsen's identity \cite{Nielsen18}):
          \begin{align}\label{Nielsen}
             H_{m+n}(x) =m!n!\sum\limits_{k=0}^{min(m,n)}\frac{(-2)^k}{k!} \frac{H_{m-k}(x)}{(m-n)!}\frac{H_{n-k}(x)}{(n-k)!}.
          \end{align}
An extension of \eqref{ROpfsum} is given by Gould and Hopper in \cite{GouldHopper62}
for the generalized Hermite polynomials
   $$
   H_m^\gamma(x,\alpha,p): = (-1)^m x^{-\alpha} e^{px^\gamma}D^m \left(x^{\alpha} e^{-px^\gamma} \right).
   $$

As an interesting extension of the real Hermite polynomials $H_{m}(x)$ are the complex Hermite polynomials $H_{p,q}(z,\bz)$, for  $z=x+iy \in\C$; $x,y\in \R$. They were considered by Ito in \cite{Ito51} in the context of complex Markov process,
         \begin{align*}
         H_{p,q}(z,\bz )=(-1)^{p+q}e^{z\bz }
        \pbz^p\pz^q \left(e^{-z \bz }\right),
         \end{align*}
where $\pz $ and $\pbz $ stand for
         $$
         \pz =  \dfrac{\partial}{\partial z} = \frac 12\left( \dfrac{\partial}{\partial x} - i  \dfrac{\partial}{\partial y}\right)
         \quad \mbox{and} \quad
         \pbz  := \dfrac{\partial}{\partial \bz } = \frac 12\left( \dfrac{\partial}{\partial x} + i  \dfrac{\partial}{\partial y}\right); \quad i=\sqrt{-1}.
         $$
 For the unity of the formulation, we define trivially $H_{p,q}(z,\bz ) = 0 $ whenever $p<0$ or $q<0$.
This class of polynomials appears naturally when investigating spectral properties of some second order differential operators of Laplacian type
 (\cite{Shigekawa87,Folland89,Thangavelu98,Wong98}) and can be connected to some classes of generalized Bargmann spaces. Several interesting features of
 $H_{p,q}(z,\bz )$ in connection of coherent states theory have been studied recently  \cite{Ali10,CoftasGazeau10,Thirulo10}. The complex Hermite polynomials $H_{p,q}(z,\bz)$ can be expressed in terms of $H_m(x)$ as
          \begin{align*}
             H_{p,q}(z,\bz ) =  p!q!\left(\frac{i}{2}\right)^{p+q}
               \sum\limits_{j=0}^{p}\sum\limits_{k=0}^{q}
               (-1)^{(q+j)}i^{j+k} \frac{H_{j+k}(x)}{j!k!}\frac{ H_{p+q-j-k}(y)}{(p-j)!(q-k)!}
          \end{align*}
 or in terms of the generalized Laguerre polynomials $L^{(\alpha)}_n(x)$ 
as
           \begin{align*}
             H_{p,q}(z,\bz ) =  (-1)^{min(p,q)} (min(p,q))! |z|^{|p-q|} e^{i(p-q)\arg(z)}  L^{(|p-q|)}_{min(p,q)}(|z|^2)
          \end{align*}
for $z=|z|e^{i\arg(z)}$ (see \cite[Eq. (2.3)]{IntInt06} where there $h_{m,p}$ means $H_{p,m}$ in ours),  so that for $q=p+k$, we get
          \begin{align}\label{Laguerre}
             H_{p,p+k}(z,\bz ) =  (-1)^{p} p! {\bar z}^{k} L^{(k)}_{p}(|z|^2).
          \end{align}
 Moreover, they constitute an orthogonal basis of the Hilbert space $L^2(\C; e^{- z\bz } dxdy)$ (see \cite{IntInt06,Gh08}):
$$ \int_{\C} H_{p,q}(z,\bz ) \overline{H_{m,n}(z,\bz )} e^{- z\bz } dxdy = \pi\delta_{pm}\delta_{qn} p!q!.$$

In course of our investigation, we establish new operational formulae involving $H_{p,q}(z,\bz )$ analogous to the Burchnall one \eqref{ROpfsum}.
In addition, we will use them in a simple way to obtain new properties satisfied by these polynomials.
Mainely, we are interested in Nielsen's identities, generating functions and Runge addition formula.

\section{Burchnall's operational formula for $ H_{p,q}(z,\bz )$} 
In order to obtain some operational formulae for $ H_{p,q}(z,\bz )$ of Burchnall type, we begin by noting that
the complex Hermite polynomials
      \begin{align}
      H_{p,q}(z,\bz ) & =(-1)^{p+q}e^{z\bz }  \pbz^{p}\pz^{q} \left(e^{-z \bz }\right)   \label{RodriguezComp}\\
                        &=(-1)^p e^{ z \bz } \pbz^{p} \left(\bz^{q} e^{- z \bz }\right)  \label{RodriguezEquivComp1}\\
                        &=(-1)^q e^{ z \bz } \pz^{q} \left(z^{p} e^{- z \bz }\right) ,   \label{RodriguezEquivComp}
      \end{align}
can be rewritten in the following equivalent forms:
      \begin{align}
      H_{p,q}(z,\bz )  &=  \left(-\pbz + z \right)^{p}(\bz^{q}) \label{OpForHer22}\\
                       &=  \left(-\pz +\bz \right)^{q}(z^{p}), \label{OpForHer2}
      \end{align}
which appears as special cases of the following

\begin{lemma}
For every sufficiently differentiable function $f$, we have
      \begin{align}
      &\left(-\pbz + z \right)^{p}(f)
       = (-1)^p e^{ z \bz } \pbz^p  \left(e^{- z \bz }f\right), \label{OpForHer11} \\
      &\left(-\pz +\bz \right)^{q}(f)
       = (-1)^q e^{ z \bz } \pz^q  \left(e^{- z \bz }f\right) .\label{OpForHer1}
      \end{align}
\end{lemma}

\begin{proof}
 We note first that \eqref{OpForHer11} can be handled in a similar way as down below for \eqref{OpForHer1}. Indeed, direct computation yields
     $
     \left(-\pz +\bz \right)\left(e^{ z \bz }g\right) = e^{ z \bz }\left(-\pz g\right).
     $
Next, by induction, we get
     $$
     \left(-\pz +\bz \right)^{q}\left(e^{ z \bz }g\right) = (-1)^q e^{ z \bz }  \pz^q  (g),
     $$
which infers \eqref{OpForHer1} for  $g=e^{- z \bz }f$.
\end{proof}

Consequently, it follows
\begin{proposition} For given positive integers $p,q$ and  every sufficiently differentiable function $f$, we have
     \begin{align}
      a) \quad & \left(-\pz +\bz \right)^{q}(f)=q!\sum\limits_{k=0}^{q}\frac{(-1)^k}{k!}
                 \frac{\bz ^{q-k}}{(q-k)!} \pz^{k} (f) ,\label{OpForHerProp}\\
      b) \quad & \left(-\pz +\bz \right)^{q}(f)
                 =q!\sum\limits_{k=0}^{q}\frac{(-1)^k}{k!} \frac{ H_{p,q-k}(z,\bz )}{(q-k)!}\pz^{k} (z^{-p} f)
                 \label{OpForHerPropBurch} .
     \end{align}
\end{proposition}

\begin{proof}
Application of the Leibnitz formula to the product function $e^{- z \bz }f$ in \eqref{OpForHer1} gives
     \begin{align*}
      \left(-\pz +\bz \right)^{q}(f)
             &= (-1)^q e^{ z \bz } \sum\limits_{k=0}^{q}\binom{q}{k} \pz^{q-k} \left(e^{- z \bz }\right) \pz^{k} (f) \nonumber \\
             &= q!\sum\limits_{k=0}^{q}\frac{(-1)^k}{k!} \frac{\bz^{q-k}}{(q-k)!} \pz^{k} (f).
     \end{align*}
Now by writing $f$ as $f=z^{p} (z^{-p}f)$ in \eqref{OpForHer1} and next applying Leibnitz formula, we obtain
      \begin{align*}
      \left(-\pz +\bz \right)^{q}(f)
              &=(-1)^q e^{ z \bz }\sum\limits_{k=0}^{q} \binom{q}{k} \pz^{q-k}
                 \left(z^{p} e^{- z \bz }\right) \pz^{k} (z^{-p} f)\\
              &\stackrel{\eqref{RodriguezEquivComp}}{=} q!\sum\limits_{k=0}^{q}\frac{(-1)^k}{k!} \frac{ H_{p,q-k}(z,\bz )}{(q-k)!}\pz^{k} (z^{-p} f).
   \end{align*}
   \end{proof}

As immediate consequence of \eqref{OpForHerProp} when taking $f=z^{p}$, we get the explicit expansion
of the complex Hermite polynomial, 
    \begin{align} \label{HerExp}
    H_{p,q}(z,\bz )&=p!q!   \sum\limits_{k=0}^{min(p,q)}\frac{(-1)^{k}}{k!} \frac{\bz^{q-k}}{(q-k)!} \frac{z^{p-k}}{(p-k)!}  .
    \end{align}

   The expression \eqref{OpForHerPropBurch} for the complex Hermite polynomial can be considered as an analogue of \eqref{ROpfsum}.
   However,  we obtain in the sequel a more appropriate Burchnall's operational formula involving $H_{p,q}(z,\bz )$.

\begin{proposition}
For given positive integers $p,q$ and  every sufficiently differentiable function $f$, we have
     \begin{align}\label{NiceBurchnall0}
 e^{z\overline{z}}\pbz^{p}\pz^q\left( e^{-z\overline{z}} f\right)
 &=(-1)^{p+q} p!q!\sum\limits_{j=0}^{p} \sum\limits_{k=0}^{q}
                 \frac{(-1)^{j+k}}{j!k!} \frac{ H_{p-j,q-k}(z,\bz ) }{(p-j)!(q-k)!} \pbz^j\pz^{k} (f).
     \end{align}
\end{proposition}

\begin{proof}
By applying repetitively the Leibnitz formula, it follows
     \begin{align*}
 e^{z\overline{z}}\pbz^{p}\pz^q\left( e^{-z\overline{z}} f\right)
 &= e^{z\overline{z}} \sum\limits_{j=0}^{p} \sum\limits_{k=0}^{q}
 (-1)^{q-k}\binom{p}{j}\binom{q}{k} \pbz^{p-j} \left( \overline{z}^{q-k}
  e^{-z\overline{z}} \right) \pbz^j\pz^{k} (f).
 \end{align*}
 But from \eqref{RodriguezEquivComp1}, i.e., $e^{z\overline{z}} \pbz^{m} \left( \overline{z}^{n}
  e^{-z\overline{z}} \right)= (-1)^{m} H_{m,n}(z,\bz)$, we deduce
      \begin{align*}
e^{z\overline{z}}\pbz^{p}\pz^q\left( e^{-z\overline{z}} f\right)
 &=  \sum\limits_{j=0}^{p} \sum\limits_{k=0}^{q}
(-1)^{j+k} \binom{p}{j}\binom{q}{k}  H_{p-j,q-k}(z,\bz) \pbz^j\pz^{k} (f).
 \end{align*}
 This completes the proof.
\end{proof}

\begin{remark}
For $f(z)= e^{-\xi z/2}$ with $\xi \in \C$, the right hand side of \eqref{NiceBurchnall0} leads to the generalized complex Hermite polynomials studied in \cite{Gh08} and suggested by a special magnetic Schrödinger operator.
\end{remark}

\begin{remark}
Using \eqref{OpForHer1} together with the established fact
$\left(-\pbz +z \right)^p\left(e^{ z \bz }g\right) = (-1)^p e^{ z \bz }\left(\pbz^p g\right)$,  we obtain
     \begin{align*}
 \left(-\pbz +z \right)^{p} \left(-\pz +\bz \right)^{q}(f)
 &=(-1)^{p+q} e^{z\overline{z}}\pbz^{p}\pz^q\left( e^{-z\overline{z}} f\right).
     \end{align*}
Therefore, the operational formula \eqref{NiceBurchnall0} can be reworded as
 \begin{align}\label{NiceBurchnall}
 \left(-\pbz +z \right)^{p} \left(-\pz +\bz \right)^{q}(f)
 &=p!q!\sum\limits_{j=0}^{p} \sum\limits_{k=0}^{q}
                 \frac{(-1)^{j+k}}{j!k!} \frac{ H_{p-j,q-k}(z,\bz ) }{(p-j)!(q-k)!} \pbz^j\pz^{k} (f).
 \end{align}
Moreover, it is interesting to observe that $H_{p,q}(z,\bz )$ can be realized also as
    \begin{align}\label{realisationHerm-pq}
    H_{p,q}(z,\bz ) = \left(-\pbz +z \right)^{p} \left(-\pz +\bz \right)^{q}\cdot (1).
    \end{align}
\end{remark}

In the next section, we are concerned with some properties, for the complex Hermite polynomials, that follow directly from the previous
 operational representations, including \eqref{OpForHer2}, \eqref{OpForHerPropBurch} and its variant \eqref{OpForHerProp} as well as \eqref{NiceBurchnall}.

\section{Related properties}

\noindent {\bf 3.1. Recurrence formulae. } Note for instance that, since $\pz$ and $\left(-\pz +\bz \right)^{q}$ commute, we get
\begin{align} \label{HerDer1}
    \pz H_{p,q}(z,\bz ) = \pz \left(-\pz +\bz \right)^{q}(z^p) = \left(-\pz +\bz \right)^{q}\pz(z^p)= p  H_{p-1,q}(z,\bz ).
    \end{align}
Similarly, we have
\begin{align} \label{HerDer2}
    \pbz H_{p,q}(z,\bz ) = \pbz \left(-\pbz + z \right)^{p}(\bz^q) = \left(-\pbz + z \right)^{p}\pbz(\bz^q)= q  H_{p,q-1}(z,\bz ) .
    \end{align}
    Thus, we see
\begin{align} \label{HerDer3}
     \left(-\pz + \bz \right)\pbz H_{p,q}   =  q  H_{p,q} ,
    \end{align}
that is $H_{p,q}$ are eigenfunctions of the second order differential operator of Laplacian type $\Delta :=\left(-\pz + \bz \right)\pbz $.
Furthermore, the polynomials $H_{p,q}$  obey the recursion relations
     \begin{align}
      a)  \quad &  H_{p+1,q}   =     z H_{p,q} -  \pbz H_{p,q}. \label{RecF1}\\
      b)  \quad &  H_{p+1,q}   =     z H_{p,q} -   q H_{p,q-1}. \label{RecF2}\\
      a')  \quad &  H_{p,q+1}   =   \bz H_{p,q} -  \pz H_{p,q}. \label{RecF11}\\
      b')  \quad &  H_{p,q+1}   =   \bz H_{p,q} -   p H_{p-1,q}. \label{RecF22}
     \end{align}
Indeed,
      using \eqref{OpForHer2}, it follows
      $$
      H_{p+1,q}(z,\bz ) = \left(-\pbz + z \right)^{1+p}(\bz^q) = \left(-\pbz + z \right)  H_{p,q}(z,\bz ) =  -\pbz H_{p,q} + z H_{p,q}.
      $$
      Substitution of \eqref{HerDer2} in the previous equality gives rise to \eqref{RecF2}.
      Similarly \eqref{OpForHer22} yields \eqref{RecF11} and so \eqref{RecF22}.
Recurrence formula \eqref{RecF11} (resp. \eqref{RecF22}) can also be obtained from \eqref{RecF1} (resp. \eqref{RecF2}) by conjugation since the polynomials $ H_{p,q}(z,\bz )$ satisfy the following symmetry relation with respect to index permutation
 $ \overline{H_{p,q}(z,\bz )}=H_{q,p}(z,\bz )$.

\quad

\noindent {\bf 3.2. Quadratic recurrence formulae. }
By writing \eqref{OpForHerProp} for the special case of $f=H_{p,s}$, we get the following first variant of Nielsen's identity
   \begin{align} \label{HerExpNielsen}
     H_{p,q+s}(z,\bz )  = p!q!   \sum\limits_{k=0}^{min(p,q)}\frac{(-1)^{k}}{k!} \frac{\bz^{q-k}}{(q-k)!}
      \frac{H_{p-k,s}(z,\bz )}{(p-k)!} .
   \end{align}
This follows using
    $   H_{p,q+s}  = \left(-\pz +\bz \right)^{q}(H_{p,s})$
combined with the fact that
    $\pz^{k} (H_{p,s})=  \dfrac{p!}{(p-k)!}H_{p-k,s}$ for $k\leq p$.
Furthermore, by considering the particular case $f=z^{p+s}$ in  \eqref{OpForHerPropBurch}, we check that
   $$ \left(-\pz +\bz \right)^{q}(z^{p+s})=
    q!\sum\limits_{k=0}^{q}\frac{(-1)^k}{k!} \frac{ H_{p,q-k}(z,\bz )}{(q-k)!}\pz^{k} (z^{s}).$$
Whence, the second  variant of the Nielsen identity reads
    \begin{align}\label{HerExpGhanmi}
     H_{p+s,q}(z,\bz )=  q!s!\sum\limits_{k=0}^{min(q,s)}\frac{(-1)^k}{k!}\frac{z^{s-k}}{(s-k)!}\frac{H_{p,q-k}(z,\bz )}{(q-k)!}.
    \end{align}

    A third variant of Nielsen's identity can be established. More precisely, we have the following proposition.

    \begin{proposition} We have
    \begin{align}\label{NielsenGhanmi}
     H_{p+s,q}(z,\bz )=   \sqrt{2}^{p+s-q}q! \sum\limits_{k=0}^{q}\dfrac{H_{s,k}\left(\dfrac{z}{\sqrt{2}},\dfrac{\bz}{\sqrt{2}}\right)}{k!} \dfrac{ H_{p,q-k}\left(\dfrac{z}{\sqrt{2}},\dfrac{\bz}{\sqrt{2}} \right)}{(q-k)!}.
    \end{align}
    \end{proposition}

    \begin{proof} The proof of \eqref{NielsenGhanmi} can be handled by writing $e^{z\bz}\left(-\pz +\bz \right)^{q}(z^{p+s} e^{- z\bz})$ in two manners. Indeed, the first one follows by taking $f= z^{p+s} e^{- z\bz}$ in \eqref{OpForHer1}. In fact,
    \begin{align*} e^{z\bz}\left(-\pz +\bz \right)^{q}(z^{p+s} e^{- z\bz})
      & = (-1)^q e^{ 2z \bz } \pz^q  \left(z^{p+s} e^{- 2z \bz }\right)\\
       &= \sqrt{2}^{q-p-s} H_{p+s,q}(\sqrt{2}z,\sqrt{2}\bz )
       \end{align*}
        upon an appropriate change of variable.
          The second one follows from \eqref{OpForHerPropBurch}, to wit
    \begin{align*}
    e^{z\bz}\left(-\pz +\bz \right)^{q}(z^{p+s} e^{- z\bz})
    &=  q!e^{z\bz}\sum\limits_{k=0}^{q}\frac{(-1)^k}{k!} \frac{ H_{p,q-k}(z,\bz )}{(q-k)!}\pz^{k} (z^{s}  e^{- z\bz})\\
    &\stackrel{\eqref{RodriguezEquivComp}}{=}  q! \sum\limits_{k=0}^{q}\frac{H_{s,k}(z,\bz )}{k!} \frac{ H_{p,q-k}(z,\bz )}{(q-k)!}.
    \end{align*}
    Thence, we have proved that
    $$  H_{p+s,q}(\sqrt{2}z,\sqrt{2}\bz ) =  \sqrt{2}^{p+s-q}q! \sum\limits_{k=0}^{q}\frac{H_{s,k}(z,\bz )}{k!} \frac{ H_{p,q-k}(z,\bz )}{(q-k)!} $$
    which gives rise to \eqref{NielsenGhanmi} by replacing $z$ by $z/\sqrt 2$.
    \end{proof}

   We have called here \eqref{HerExpNielsen} (resp. \eqref{HerExpGhanmi}) by the first (resp. second) variant of Nielsen identity, since
we can write it as a weighted sum of a product of the same polynomials, according to the fact that $\bz^{q-k}=H_{0,q-k}(z,\bz )$
(resp. $z^{s-k}=H_{s-k,0}(z,\bz )$). We reserve the appellation Nielsen's identity to the following

\begin{proposition} We have the following quadratic recurrence formula 
      \begin{align}\label{NielsenR}
{H_{p+m,q+n}(z,\bz )}{}
 &=p!q!m!n! \sum\limits_{j=0}^{min(p,n)} \sum\limits_{k=0}^{min(q,m)}
     \frac{(-1)^{j+k}}{j!k!} \frac{ H_{p-j,q-k}(z,\bz ) }{(p-j)!(q-k)!} \frac{H_{m-k,n-j}(z,\bz )}{(m-k)!(n-j)!} .
     \end{align}
    \end{proposition}

    \begin{proof}
    By taking $f=H_{m,n}(z,\bz )$ in \eqref{NiceBurchnall}, we obtain
     \begin{align*}
 \left(-\pbz +z \right)^{p} & \left(-\pz +\bz \right)^{q}(H_{m,n}(z,\bz ))
 =p!q!\sum\limits_{j=0}^{p} \sum\limits_{k=0}^{q}
                 \frac{(-1)^{j+k}}{j!k!} \frac{ H_{p-j,q-k}(z,\bz ) }{(p-j)!(q-k)!}\pbz^j\pz^{k} (H_{m,n}(z,\bz )).
     \end{align*}
 Therefore the result \eqref{NielsenR} follows upon making use of
 $$
 \left(-\pbz +z \right)^{p} \left(-\pz +\bz \right)^{q}(H_{m,n}(z,\bz ))= H_{p+m,q+n}(z,\bz )
 $$
 together with the fact that
 $$
 \pbz^j\pz^{k} (H_{m,n}(z,\bz ))= m!n! \frac{H_{m-k,n-j}(z,\bz )}{(m-k)!(n-j)!} $$
 for $k\leq m$, $j\leq n$ and $\pbz^j\pz^{k} (H_{m,n}(z,\bz ))=0$ otherwise.
    \end{proof}

As a special case of \eqref{NielsenR} when taking $m=q$ and $n=p$ (keeping in mind the fact that $ \overline{H_{m,n}(z,\bz )}= H_{n,m}(z,\bz )$), we can state the following

       \begin{corollary} We have
       \begin{align}\label{NielsenRCor}
 H_{p+q,p+q}(z,\bz )
 &=(p!q!)^2 \sum\limits_{j=0}^{p} \sum\limits_{k=0}^{q}
                 \frac{(-1)^{j+k}}{j!k!} \frac{ |H_{p-j,q-k}(z,\bz )|^2 }{((p-j)!(q-k)!)^2}  .
     \end{align}
          \end{corollary}


\noindent {\bf 3.3. Generating functions. }
\begin{proposition} We have the following generating functions:
        \begin{align}
         a) & \quad \sum\limits_{p=0}^{+\infty}\frac{u^p}{p!}H_{p,q}(z,\bz ) = (\bz- u)^q e^{ u z } .
        \label{GenFctComp0} \\
        a') & \quad \sum\limits_{q=0}^{+\infty}\frac{v^q}{q!}H_{p,q}(z,\bz ) = (z-v)^p e^{ v\bz } .
        \label{GenFctComp11} \\
         b) & \quad
        \sum\limits_{p,q=0}^{+\infty}\frac{u^p}{p!}\frac{v^q}{q!}H_{p,q}(z,\bz )=e^{ uz + v\bz -uv}.
        \label{GenFctComp1}
        \end{align}
        These equalities are valid for all $u$, $v$, $z$ complex.
\end{proposition}

\begin{proof}
The generating function $a)$ is in fact the conjugate counterpart of $a')$. To prove $a')$, we begin by writing
          \begin{align*}
       \sum\limits_{q=0}^{+\infty}\frac{v^q}{q!}\left(-\pz +\bz \right)^{q}(f)
                                      =e^{-v\pz + v \bz }(f)  
                                      =e^{ v\bz }e^{-v\pz }(f).
       \end{align*}
Now, since $e^{-v\pz }(z^{p} ) = \sum\limits_{k=0}^{p}  \dbinom{p}{k} (-v)^k z^{p-k}$, it follows
       \begin{align*}
       \sum\limits_{q=0}^{+\infty}\frac{v^q}{q!} H_{p,q}(z,\bz ) 
        = e^{ v\bz } \sum\limits_{k=0}^{p} \binom{p}{k} (-v)^k z^{p-k}
        =(z-v)^p e^{ v\bz }.
      \end{align*}
This is exactly \eqref{GenFctComp11} from which we deduce
    \begin{align*}
    \sum\limits_{p,q=0}^{+\infty}\frac{u^p}{p!}\frac{v^q}{q!} H_{p,q}(z,\bz )
       =  e^{ v\bz }\sum\limits_{p,q=0}^{+\infty}\frac{u^p}{p!}(z-v)^p
             =e^{ uz + v\bz -uv}.
    \end{align*}
We conclude the proof by providing upper bound of the complex Hermite polynomials
and next discussing convergence conditions of the involved series.
In fact, for fixed $p$ and large $q$, say $q=p+k$, we get from \eqref{Laguerre} that
\begin{align}
 \left|H_{p,p+k}(z,\bz ) \right|  =  p!  \left|z\right|^{k}  \left| L^{(k)}_{p}(|z|^2)\right|
 \leq (p+k)!  \frac{|z|^k}{k!}e^{|z|^2/2}, \label{estimate}
\end{align}
where the obtained upper bound in \eqref{estimate} follows
using the classical global uniform estimate for the generalized Laguerre polynomials due to Szeg\"o (see \cite{Szego75,Lewandowski98}):
$$
\left| L^{(\alpha)}_{n}(x)\right| \leq L^{(\alpha)}_{n}(0) e^{x/2} =\frac{\Gamma(n+\alpha+1)}{n!\Gamma(\alpha+1)} e^{x/2}
$$
for $\alpha, x \geq 0$ and $n=0,1,2, \cdots$. Therefore,
one checks easily the following
   \begin{align*}
            \left|\frac{v^{(p+k)}}{(p+k)!} H_{p,p+k}(z,\bz ) \right|
            &\leq  \frac{|v z|^k}{k!} \left(|v|^p e^{|z|^2/2}\right).%
          \end{align*}
This shows that the series \eqref{GenFctComp11}, in the variable $v$, converges absolutely and uniformly on any compact set of $\C$, for every complex $z$. For the series \eqref{GenFctComp1}, the convergence is then immediate and holds for every $u,v,z \in \C$.
\end{proof}

\begin{remark}
We do not claim the representation \eqref{GenFctComp1} is new, however we have difficulty to find a rigorous simpler proof of it in the literature.
\end{remark}

\begin{corollary} We have
        \begin{align*}
        \sum\limits_{p,q=0}^{+\infty}\frac{{\bar z}^pz^q}{p!q!}H_{p,q}(z,\bz )=e^{ z \bz}.
        \end{align*}
Furthermore, for every fixed integers $m,n\geq 0$, we have the identities
         \begin{align*}
          \sum\limits_{p=0}^{+\infty}\frac{\bz^p}{p!}H_{p,n}(z,\bz ) = 0
          \quad \mbox{and} \quad
          \sum\limits_{q=0}^{+\infty}\frac{z^q}{q!}H_{m,q}(z,\bz ) = 0  .
        \end{align*}
\end{corollary}

According to the referee suggestion, it would be more interesting, in obtaining reproducing kernels using coherent states, if one can provide the sum
         \begin{align*}
        \sum\limits_{q=0}^{+\infty}\frac{H_{p,q}(z,\bz ) \overline{H_{p,q}(w,\bar w )} }{p!q!}.
        \end{align*}
        This is contained in \cite[Eq. (2.6)]{AIM00} (see also \cite[p. 405]{IntInt06}). However, the following result gives a closed form of such sum in a more general context.

\begin{proposition} For every $z,w\in\C$ and fixed positive integers $p,m$, we have
         \begin{align}
        \sum\limits_{q=0}^{+\infty}\frac{H_{p,q}(z,\bz ) \overline{H_{m,q}(w,\bar w )} }{q!} =
         H_{p,m}(z-w,\overline{z-w}) e^{w\bz}.
        \label{SumRepKer1}
        \end{align}
\end{proposition}

\begin{proof}
Multiplication of both sides of \eqref{GenFctComp11} by $e^{-w \bar w}$ and next differentiating the obtained formula $m$-times with respect to $w$ infers
$$
\sum\limits_{q=0}^{+\infty}\frac{1}{q!}H_{p,q}(z,\bz ) \partial_w^m\left( w^q e^{-w \bar w}\right) = \partial_w^m\left((z-w)^p e^{ w(\bz-\bar w) }\right).
$$
Now, in view of \eqref{RodriguezEquivComp}, i.e., $\partial_w^m\left( w^q e^{-w \bar w}\right) =(-1)^m H_{q,m}(w,\bar w )e^{-w \bar w} $, we obtain
\begin{align}
\sum\limits_{q=0}^{+\infty}\frac{H_{p,q}(z,\bz )H_{q,m}(w,\bar w)}{q!} = (-1)^m  e^{w \bar w} \partial_w^m\left((z-w)^p e^{ w(\bz-\bar w) }\right).
\label{PfRepKer}
\end{align}
The right hand side of the previous equation can be rewritten as
\begin{align*}
(-1)^m  e^{w \bar w} \partial_w^m\left((z-w)^p e^{ w(\bz-\bar w) }\right)
&=  (-1)^m   e^{w \bz} \left[ e^{ (z-w)(\overline{z-w}) }\partial_w^m\left((z-w)^p e^{ -(z-w)(\overline{z-w}) }\right) \right]\\
&=    e^{w \bz} \left[ (-1)^m e^{ \xi\bar\xi }\partial_\xi^m\left(\xi^p e^{ -\xi\bar\xi }\right)\Big|_{\xi=z-w} \right]\\
&\stackrel{\eqref{RodriguezEquivComp}}{=}
  H_{p,m}(z-w,\bz-\bar w )  e^{w \bz}.
\end{align*}
Replacing this in \eqref{PfRepKer} leads to the desired result.
\end{proof}

\begin{remark} As a particular case of \eqref{SumRepKer1}, when taking $p=m$, we get
        \begin{align}
        \sum\limits_{q=0}^{+\infty}\frac{H_{p,q}(z,\bz ) \overline{H_{p,q}(w,\bar w )} }{p!q!} =
         L^{(0)}_p(|z-w|^2)  e^{w\bz},
        \label{SumRepKer2}
        \end{align}
        since in this case $ H_{p,p}(\xi,\bar\xi) =  p! L^{(0)}_p(|\xi|^2)  $ (see \eqref{Laguerre} with $k=0$).
        Furthermore, by rewriting the sum in the left hand side of \eqref{SumRepKer2} as
        $ \sum\limits_{q=0}^{+\infty} =  \sum\limits_{q=0}^{p-1} +  \sum\limits_{q=p}^{+\infty}$,
        and next taking $z=w$ and $n=q-p$ for $q\geq p$, one deduces the following
         \begin{align}
        \sum\limits_{n=0}^{+\infty}\frac{|H_{p+n,p}(z,\bz )|^2 }{(p+n)!p!} =
        \frac{(-1)^p}{p!} e^{z\bz} - \sum\limits_{q=0}^{p-1}\frac{|H_{p,q}(z,\bz )|^2}{p!q!}.
        \label{SumRepKer3}
        \end{align}
\end{remark}

\begin{proposition} We have the following summation formula:
        \begin{align}
        \sum\limits_{p,q,r,s=0}^{+\infty}\frac{ H_{p,q}(z,\bz ) H_{p,s}(u,\bar u ) H_{q,r}(v,\bar v )}{p!q!r!s!}
        = e^{ z(u+1)} e^{\bz (v+1)}e^{-(u+v+uv+1)}.
        \label{GenFctComp2}
        \end{align}
\end{proposition}

\begin{proof} We make use of \eqref{OpForHer2}  to write the left hand side of \eqref{GenFctComp2} as
     \begin{align*}
        \sum\limits_{p,q,r,s=0}^{+\infty} \frac{1}{p!q!r!s!} &H_{p,q}(z,\bz ) H_{p,s}(u,\bar u ) H_{q,r}(v,\bar v )\\
        &=\sum\limits_{r,s=0}^{+\infty}  \frac{1}{r!s!}(-\partial_v +\bar v)^r (-\partial_u +\bar u)^s
        \left( \sum\limits_{p,q=0}^{+\infty} \frac{u^p}{p!} \frac{v^q}{q!} H_{p,q}(z,\bz )\right)
        \end{align*}
By applying \eqref{GenFctComp1}, we obtain
         \begin{align*}
        \sum\limits_{p,q,r,s=0}^{+\infty}\frac{1}{p!q!r!s!} &H_{p,q}(z,\bz ) H_{p,s}(u,\bar u ) H_{q,r}(v,\bar v )\\
        & 
        =\sum\limits_{r,s=0}^{+\infty}  \frac{1}{r!s!}(-\partial_v +\bar v)^r (-\partial_u +\bar u)^s
         \left( e^{uz + v\bz -uv}\right)\\
        &=\sum\limits_{r,s=0}^{+\infty}  \frac{1}{r!}(-\partial_v +\bar v)^r  \left( \frac{(z-v)^s}{s!}e^{ uz + v\bz - uv}\right)\\
        &=\sum\limits_{r=0}^{+\infty}  \frac{1}{r!}(-\partial_v +\bar v)^r  \left( e^{z-v}e^{ uz + v\bz - uv}\right)\\
        &=\sum\limits_{r=0}^{+\infty}  \frac{(-1+\bz -u)^r}{r!}  e^{z-v}e^{ uz + v\bz - uv} \\
        &=e^{ z(u+1)} e^{\bz (v+1)}e^{-(u+v+uv+1)}.
        \end{align*}
\end{proof}

\begin{remark}
The special case of $r=s=0$ (in \eqref{GenFctComp2}) leads to the generating relation \eqref{GenFctComp1}.
\end{remark}

\noindent {\bf 3.4. Runge's addition formula. }

  \begin{proposition} We have the following addition formula
      \begin{align} \label{add-forRealDHP}
      H_{p,q}\left( z+w, \bz + \bw \right)=p!q!\left(\frac 1{2}\right)^{(p+q)/2}  \sum_{j=0}^p  \sum\limits_{k=0}^{q}
     \frac{H_{j,k}(\sqrt{2}z,\sqrt{2}\bz)}{k!j!} \frac{H_{p-j,q-k}(\sqrt{2}w,\sqrt{2}\bw)}{(q-k)!(p-j)!} .
      \end{align}
 \end{proposition}

\begin{proof}
Set $ A_z := -\frac{\partial}{\partial (\sqrt 2\overline{z})}+ (\sqrt 2 z) $ and
$ A_{\overline{z}} := -\frac{\partial}{\partial (\sqrt 2z)}+ (\sqrt 2 \overline{z}) $,
and note that
  \begin{align}\label{Hermitesrt}
   A_{\overline{z}}^m A_z^n \cdot (1) =  H_{m,n}(\sqrt{2}z,\sqrt{2}\overline{z}).
  \end{align}
 Note also that $H_{p,q}\left( z+w, \bz + \bw \right)$ can be written in terms of $A_z$ and $A_w$ as
  \begin{align*}
  H_{p,q}\left( z+w ,\overline{z+w} \right)
     & =  \left(-\frac{d}{d(\overline{z+w})}+ (z+w)\right)^p\left(-\frac{d}{d(z+w)}+ (\overline{z+w})\right)^q.(1)\\
     & = \left(\frac12\right)^{(p+q)/2} \left\{  \left(A_z + A_w \right)^p\left(A_{\overline{z}}+ A_{\overline{w}} \right)^q \right\} .(1).
  \end{align*}
Now, since the involved differential operators $A_z$, $A_{\overline{z}}$, $A_w$ and $A_{\overline{w}}$ commute, the binomial formula yields
   \begin{align*}
  H_{p,q}\left( z+w ,\overline{z+w} \right)
             & = \left(\frac12\right)^{(p+q)/2} \left( \sum_{j=0}^p \sum_{k=0}^q \binom{p}{j} \binom{q}{k} A_z^j
             A_w^{p-j}    A_{\overline{z}}^k  A_{\overline{w}}^{q-k} \right) .(1)\\
              & = \left(\frac12\right)^{(p+q)/2} \sum_{j=0}^p \sum_{k=0}^q \binom{p}{j} \binom{q}{k}
            \left( A_z^j  A_{\overline{z}}^k .(1) \right) \left(  A_w^{p-j}  A_{\overline{w}}^{q-k}.(1) \right).
             \end{align*}
  Thence, we obtain the asserted result according to \eqref{Hermitesrt}.
 \end{proof}

\begin{remark}
An alternative way to prove \eqref{add-forRealDHP} is to write $H_{p,q}\left( z+w, \bz + \bw \right)$ in terms of the operator $ A_z $ as
\begin{align*}
  H_{p,q}\left( z+w, \bz + \bw \right) & = \left(\frac 1{2}\right)^{(p+q)/2} \sum_{j=0}^p     \binom{p}{j}\left( A_z + A_w \right)^q .((\sqrt{2}z)^j (\sqrt{2}w)^{p-j})
  \end{align*}
and next applying the binomial formula, keeping in mind that the operators $A_z$ and $A_w$ commute and satisfy
  $$  A_z^r(z^s) = {2}^{-s/2}H_{s,r}(\sqrt{2}z,\sqrt{2}\overline{z}).$$
\end{remark}




\section*{Acknowledgements}
The author is thankful to the anonymous referees for their valuable suggestions for improving the presentation of the paper, and Professor A. Intissar for valuable discussions.


\begin{thebibliography}{32}
\bibitem[1]{AIM00}
N. Askour, A. Intissar, Z. Mouayn,
        {\em Explicit formulae for reproducing kernels of generalized Bargmann spaces on $\C^n$},
        J. Math. Phys. 41 (2000) 3057--3067.
\bibitem[2]{Ali10}
S.T. Ali, F. Bagarello, G. Honnouvo,
        {\em Modular structures on trace class operators and applications to Landau levels},
        { J. Phys. A}, 43, no. 10 (2010) 105202, 17 pp.
\bibitem[3]{Al-Salam61} 
W.A. Al-Salam,
         {\em A generalized Turán expression for the Bessel functions},
         { Amer. Math. Monthly} 68 (1961) 146--149.
\bibitem[4]{Al-SalamIsmail75} 
W.A. Al-Salam, M.E.H. Ismail,
         {\em Some operational formulas},
         { J. Math. Anal. Appl.} 51 (1975) 208--218.
\bibitem[5]{Burchnall41} 
 J. L. Burchnall,
         {\em A note on the polynomials of Hermite},
         { Quart. J. Math.}, Oxford Ser. 12 (1941) 9-11.
\bibitem[6]{Burchnall51}
  J.L. Burchnall,
         {\em The Bessel polynomials},
         { Canadian J. Math.} 3 (1951) 62--68.
\bibitem[7]{Carlitz60} 
L. Carlitz,
         {\em A note on the Laguerre polynomials},
         { Michigan Math. J.} 7 (1960) 219--223.
\bibitem[8]{Chatterjea63a} 
S.K. Chatterjea,
         {\em Operational formulae for certain classical polynomials I},
         {Quart. J. Math.} Oxford Ser. (2) 14 (1963) 241--246.
\bibitem[9]{ChatterjeaSrivastava93} 
S.K. Chatterjea; H.M. Srivastava,
             {\em A unified presentation of certain operational formulas for the Jacobi and related polynomials},
             { Appl. Math. Comput.} 57, no. 1 (1993) 77--95.
\bibitem[10]{CoftasGazeau10} 
N. Cotfas, J-P. Gazeau, K. Górska,
        {\em Complex and real Hermite polynomials and related quantizations},
        { J. Phys. A}  43, no. 30 (2010) 305304, 14 pp.
\bibitem[11]{Folland89} 
G.B. Folland,
        {\em Harmonic analysis in phase space},
        Princeton university press, New Jersey, 1989.
\bibitem[12]{Gh08} 
A. Ghanmi,
        {\em A class of generalized complex Hermite polynomials},
        { J. Math. Anal. App.}  340 (2008) 1395-1406.
\bibitem[13]{GouldHopper62} 
H.W. Gould, A.T. Hopper,
        {\em Operational formulas connected with two generalizations of Hermite polynomials},
        { Duke Math. J.} 29 (1962) 51-63.
\bibitem[14]{Hermite1864-1908} 
 C. Hermite,
        {\em Sur un nouveau développement en série des fonctions},
        Compt. Rend. Acad. Sci. Paris 58, t. LVIII (1864) 94-100 et 266-273
        ou Oeuvres complètes, tome 2. Paris, p. 293-308, 1908.
\bibitem[15]{IntInt06} 
A. Intissar,  A. Intissar,
        {\em Spectral properties of the Cauchy transform on $L^2(\C;e^{-|z|^2}d\lambda)$},
        { J. Math. Anal. Appl.} 313, no 2 (2006) 400-418.
\bibitem[16]{Ito51} 
K. It\^o,
        {\em Complex multiple Wiener integral},
        Jap. J. Math. 22 (1952) 63-86.
\bibitem[17]{Kaanoglu12} 
C. Kaanoglu,
           {\em Operational formula for Jacobi-Pineiro polynomials},
           { J. Comput. Anal. Appl.} 14, no. 6 (2012) 1096--1102.
\bibitem[18]{Kampe23} 
J. Kampé de Fériet,
        {\em Sur une formule d'addition des polynômes d'Hermite},
        Volume 2, {\it Mathematisk-fysiske Meddelelser}.  Det Kgl. Danske Videnskabernes Selskab, Lunos, 10 pages, 1923.
\bibitem[19]{KarandeThakare73} 
B.K. Karande, N.K. Thakare,
         {\em Remarks on the operational formulas for orthogonal polynomials},
         {  Mat. Vesnik} 10 (25) (1973) 139--144.
\bibitem[20]{Lewandowski98}
Z. Lewandowski, J. Szynal,
        {\em An upper bound for the Laguerre polynomials},
        J. Comput. Appl. Math. 99, no. 1-2 (1998) 529–533.
\bibitem[21]{Nielsen18}
 N. Nielsen,
        {\em Recherches sur les polynômes d'Hermite},
        Volume 1, { Mathematisk-fysiske meddelelser}.  Det Kgl. Danske Videnskabernes Selskab, 79 pages, 1918.
\bibitem[22]{Rainville71} 
E.D. Rainville,
        {\em Special functions}, Chelsea Publishing Co., Bronx, N.Y., 1971.
\bibitem[23]{PatilThakare78} 
K.R. Patil, N.K. Thakare,
          {\em Operational formulas and generating functions in the unified form for the classical orthogonal polynomials},
          { Math. Student} 45, no. 1 (1978) 41--51.
\bibitem[24]{PurohitRaina09} 
S.D. Purohit; R.K. Raina,
            {\em On certain operational formula for multivariable basic hypergeometric functions},
            { Acta Math. Univ. Comenian.} 78, no. 2 (2009) 187--195.
\bibitem[25]{Runge14}  
C. Runge,
        {\em Über eine besondere Art von Intergralgleichungen},
        { Math. Ann.} 75 (1914) 130-132.
\bibitem[26]{Singh65} 
R.P. Singh,
         {\em Operational formulae for Jacobi and other polynomials},
         { Rend. Sem. Mat. Univ. Padova} 35 (1965) 237--244.
\bibitem[27]{Shigekawa87}  
I. Shigekawa,
        {\em Eigenvalue problems of Schrödinger operator with magnetic field on compact Riemannian manifold},
        { J. Funct. Anal.} 75 (1987) 92-127.
\bibitem[28]{Szego75} 
G. Szeg\"o,
        {\em Orthogonal polynomials. Fourth edition},
        American Mathematical Society, Providence, R.I., 1975.

\bibitem[29]{Thangavelu93} 
S. Thangavelu,
        {\em Lectures on Hermite and Laguerre Expansions},
        Princeton University Press, 1993.
\bibitem[30]{Thangavelu98} 
S. Thangavelu,
        {\em Harmonic Analysis on the Heisenberg Group, Birkhauser}, 1998.
\bibitem[31]{Thirulo10} 
K. Thirulogasanthar, G. Honnouvo, A. Krzyzak,
        {\em Coherent states and Hermite polynomials on quaternionic Hilbert spaces},
         {J. Phys. A}  43, no. 38 (2010) 385205, 13 pp.
\bibitem[32]{Wong98} 
M.W. Wong,
        {\em Weyl Transforms}, Springer-Verlag, 1998.
\end{thebibliography}
\end{document}